\documentclass[12pt,a4paper,reqno]{amsart}

\usepackage[active]{srcltx}
\usepackage{enumitem}
\usepackage{float}
\usepackage{amssymb}

\renewcommand{\dim}{\operatorname{dim}}
\newcommand{\Aut}{\operatorname{Aut}}
\newcommand{\Hom}{\operatorname{Hom}}

\newcommand{\M}{{\mathcal M}}
\newcommand{\Q}{{\mathbb Q}}
\newcommand{\G}{\mathcal{G}}
\newcommand{\E}{{\mathcal E}}
\newcommand{\A}{{\mathcal A}}

\newtheorem{theorem}{Theorem}[section]
\newtheorem{lemma}[theorem]{Lemma}
\newtheorem{corollary}[theorem]{Corollary}
\newtheorem{proposition}[theorem]{Proposition}

\newtheorem{remark}[theorem]{Remark}

\newtheorem{definition}[theorem]{Definition}

\title[Finite groups as groups of self-homotopy equivalences]{{Every finite group is the group of self-homotopy equivalences of an elliptic space}}

\author{Cristina ~Costoya}
\address[C.~Costoya]{Departamento de Computaci\'on, \'Alxebra,
Universidade da Coru{\~n}a, Campus de Elvi{\~n}a, 15071  A Coru{\~n}a, Spain.}
\email[C.~Costoya]{cristina.costoya@udc.es}

\author{Antonio ~Viruel}
\address[A.~Viruel]{
Departamento de {\'A}lgebra, Geometr{\'\i}a y Topolog{\'\i}a,
Universidad de M{\'a}\-la\-ga, Campus de Teatinos, 29071 M{\'a}laga,
Spain.}
\email[A.~Viruel]{viruel@uma.es}

\thanks{First author is partially  supported by Ministerio
de Ciencia e Innovaci\'on (European FEDER support included), grant
MTM2009-14464-C02-01, and by Xunta Galicia grant Incite 09 207 215
PR}
\thanks{Second author is partially supported by Ministerio
de Ciencia e Innovaci\'on (European FEDER support included), grant
MTM2010-18089, and JA grants FQM-213 and P07-FQM-2863}
\begin{document}

\begin{abstract}
We prove that every finite group $G$ can be realized as the group of self-homotopy equivalences of infinitely many elliptic spaces $X$. To construct those spaces we introduce a new technique which leads, for example, to the existence of infinitely many inflexible manifolds. Further applications to representation theory will appear in a separate paper.
\end{abstract}

\maketitle

\section{Introduction}\label{intro}
For simply connected CW-complexes $X$ of finite type, we are
interested in the  group of homotopy classes of self-equivalences,
$\E (X)$,  and the realizability problem  for  groups.
Namely, if a given group $G$ can appear as the group $\E(X)$ for
some space $X$.  This problem has been placed as the first
to solve in
\cite{A1}, being around for over $50$ years and recurrently appearing in surveys and lists of open problems about self-homotopy equivalences
\cite{A2, F2, ka1, ka3, Ru}.  The difficulty of this question relies on the fact that techniques used so far are specific to certain groups  \cite{BM, BM2, FF, Ma, Oka},  and  have not proved fruitful when addressing this problem in general.

Apart from the group of automorphisms of a group $\pi$, $\Aut (\pi)$,  which is isomorphic to $ \E (K (\pi, n)) $  for an Eilenberg-MacLane space $K(\pi, n)$,  there is no global picture in this context.  A special mention deserves the cyclic group of order $2$, which is the group of automorphisms of the cyclic group of order $3$, hence it can be realized as $\E(K(\mathbb Z_3, n))$. Arkowitz and Lupton show that, moreover, it is  the group of self-homotopy equivalences of a  rational space,  pointing out  the surprising appearance of a finite group in rational homotopy theory and, raising the question of when finite groups can be realized by rational spaces  \cite{AL2}.

In this paper, we give a complete answer to the realizability problem for finite groups.

\begin{theorem}\label{thmgroup}
Every finite group $G$ can be realized as the group of
self-homotopy equivalences of infinitely many (non homotopy
equivalent) rational elliptic spaces $X$.
\end{theorem}

To build up those spaces, we introduce a general method which we hope can be useful for obtaining
examples with interesting properties in subjects of different nature. For instance, it appears to
produce differential manifolds related with a question of Gromov,  as it is mentioned below (see also Section \ref{inflexibledgca}).  Indeed, we construct a contravariant functor  from a subcategory of finite graphs to the homotopy category of differential graded commutative algebras
whose cohomology is $1$-connected and of finite type.
Then, the geometric realization functor of
Sullivan \cite{Su} gives the equivalence of categories between the homotopy category of minimal Sullivan algebras, and the homotopy category of rational simply connected spaces of finite
type.

We remark that by dropping the requirement on the finiteness type of the differential graded algebras, our method can be extended to infinite, locally finite graphs. This is a subtle and technical point that is handled in \cite{CV3} where this extended version of our techniques is used to obtain an isomorphism criteria for a large class of groups, having thus consequences in representation theory.

In this paper,  we prove the following theorem.
\begin{theorem}\label{thmgraph}
Let $\G$ be a finite connected graph with more than one vertex.  Then, there
exists an elliptic minimal Sullivan algebra $\M_\G$ such that the group of automorphisms
of $\G$ is realizable by the group of self-homotopy equivalences
of $\M_\G$.
\end{theorem}

Our idea of using graphs has its origin on the following classical result (\cite{Frucht1, Frucht2}).
\begin{theorem}[Frucht, 1939]\label{Frucht}{Given a finite group $G$, there exist
infinitely  many non-isomorphic  connected (finite) graphs $\G$
whose automorphism group is isomorphic to $G$. }
\end{theorem}

Because the equivalence given by the geometric realization functor of Sullivan,  Theorem~\ref{thmgroup} follows directly from Theorem
\ref{thmgraph} and Theorem \ref{Frucht} (see
Proposition \ref{rem}).  Applying Theorem~\ref{thmgroup} to the trivial group, we supply a partial answer
 to Problem 3 in \cite{ka3}. This problem consists on
determining spaces, which were thought to be quite rare \cite{ka1}, with a trivial group of self-homotopy equivalences, the so-called homotopically-rigid spaces.
\begin{corollary}
There exist infinitely many rational spaces that are homotopically-rigid.
\end{corollary}

Recall that in homotopy theory, naive dichotomy \cite{F1}
classifies spaces in either elliptic or hyperbolic.
 Ellipticity is a very severe restriction on a
space $X$ being remarkable that many of the spaces which play an
important role in geometry are rationally elliptic. In particular the rational cohomology of $X$ satisfies Poincar\'e duality
\cite{Hal} and, with extra hypothesis on the dimension of the fundamental class, $X$ has the rational homotopy type of a simply connected manifold (\cite{Ba, Su}). Indeed, spaces in Theorem \ref{thmgroup} can be chosen to have the rational homotopy type of a special class of simply connected manifolds called inflexible.
A manifold $M$ is called inflexible if all its self-maps  have degree $-1, 0, \text{ or } 1$. The work of Crowley-L{\"o}h \cite{CL} relates the existence of inflexible $d$-manifolds
with the existence of functorial semi-norms on singular homology in degree $d$ that are positive and finite on certain homology classes of simply connected spaces, solving in the negative a question raised by Gromov \cite{Gromov}. Following those ideas we prove the following.

\begin{theorem}\label{realizationinflex} Any finite group $G$ can be realized by the group of self-homotopy equivalences of the rationalization of an inflexible manifold $M$.
\end{theorem}

\begin{corollary}\label{inflex}
For every $n \in{\mathbb N}$, $n>1$, there are functorial semi-norms on singular homology in degree $d=415 + 160n$ that are positive and finite on certain homology classes of simply connected spaces.
\end{corollary}

This paper is organized as follows. In Section \ref{geomreal}, for any finite connected graph $\G$, we construct an elliptic minimal Sullivan algebra $\M_\G $ such that its group of self-homotopy equivalences, $ \E (\M_\G)$, is isomorphic to the automorphims group of the graph, $\Aut (\G)$. The construction, restricted to a suitable category of graphs, gives a contravariant faithful functor injective on objects (see Remark \ref{contrafunctor}). This algebra $\M_\G$ is inspired on \cite{AL2} where some examples of minimal Sullivan algebras, verifying that the monoid of homotopy classes of self-maps is neither trivial nor infinite, are constructed,  thus disproving a conjecture of Copeland-Shar \cite{C-S}. Our construction gives infinitely many examples of this nature (see Theorem \ref{sullgraph}). In Section \ref{inflexibledgca}, we upgrade our construction in order for it to be the rational homotopy type of an inflexible manifold $M$.

For the basic facts about graphs, we refer to \cite{bollo}. Only simple graphs $\G = (V, E)$ are considered. This means that they do not have loops and they are not directed, that is, for any $u $ vertex in $V$, the edge $(u, u)$
is not in  $E$ and, if an edge $(v, w) $ is in $E$,  then $(w, v)$ is also in $E$.  We refer to \cite{FHT2} for basic facts in rational homotopy theory. Only simply connected $\Q$-algebras of finite type are considered. If $W$ is a graded rational vector space, we write $\Lambda W$ for the free commutative graded algebra on $W$. This is a symmetric algebra on $W^{\text{even}}$  tensored with an exterior algebra on $W^{\text{odd}}$. A Sullivan algebra is a commutative differential graded algebra which is free as commutative graded algebra on a simply connected graded vector space $W$ of finite dimension in each degree. It is minimal if in addition $d(W) \subset \Lambda^{\geq 2}W$.  A Sullivan algebra is pure if  $d= 0$ on $W^{\text{even}}$ and $d (W^{\text{odd}}) \subset W^{\text{even}}$.

\section{From graphs to Elliptic Sullivan Algebras}\label{geomreal}
Ellipticity for a Sullivan algebra $(\Lambda W, d)$ means that
{both $W$  and $H^\ast (\Lambda W)$ are finite-dimensional.
Hence, the cohomology is  a Poincar\'e duality algebra \cite{Hal}. One can easily compute the degree of its fundamental class  (a fundamental class of a Poincar\'e duality algebra $H = \Sigma_{i=0}^n H^i$ is a generator of $H^n$,  $n$ is called the formal dimension of the algebra) by the formula:
\begin{equation}\label{dimensiontop}
\overset{p}{\underset{i=1}{\Sigma}}(\deg x_i) - \overset{q}{\underset{j=1}{\Sigma}}(\deg y_j -1)
\end{equation}
where $\deg x_i$ are the degrees of the elements on a basis of $W^{\text{odd}}$ and $\deg y_j$ of a basis of $W^{\text{even}}$.

\begin{definition} For a finite  connected graph $\G = (V, E)$ with more than one vertex,
we define the minimal Sullivan algebra associated to $\G$ as
$$\M_\G= \Big(\Lambda(x_1,x_2,y_1,y_2,y_3,z)\otimes\Lambda(x_v,z_v \mid
v\in V),d\Big)$$
 where degrees and differential are described by
 \begin{alignat*}{2}
&\deg x_1 = 8, \qquad \qquad&d(x_1)&=0\\
&\deg x_2 = 10,& d(x_2)&=0\\
&\deg y_1 = 33,& d(y_1)&=x_1^3x_2\\
&\deg y_2= 35,& d(y_2)&=x_1^2x_2^2\\
&\deg y_3 = 37,& d(y_3)&=x_1x_2^3\\
&\deg x_v = 40,& d(x_v)&=0\\
&\deg z= 119,& d(z)&=y_1y_2x_1^4x_2^2-y_1y_3x_1^5x_2+y_2y_3x_1^6+x_1^{15}+x_2^{12}\\
&\deg z_v = 119, & d(z_v)&=x_v^3+\sum_{(v,w)\in E}x_vx_wx_2^4.
\end{alignat*}
\end{definition}

\begin{lemma}\label{elliptic}
The constructed
$\M_\G =(\Lambda W, d)$ is an elliptic minimal Sullivan algebra of formal dimension  $n = 208 + 80 \vert V \vert$, with $\vert V \vert$ the order of the graph.
\end{lemma}
\begin{proof}
We need to prove that the cohomology of $(\Lambda W, d)$ is finite-dimensional.
Instead, we prove that the cohomology of the pure Sullivan algebra associated with $(\Lambda W, d)$
is finite-dimensional, which is an equivalent condition \cite[Proposition 32.4]{FHT2}.

The pure Sullivan algebra associated with  $(\Lambda W, d)$, denoted by $(\Lambda W, d_\sigma)$, is determined by its differential which is described by
\begin{alignat*}{3}
& d_\sigma(x_1) = 0\qquad \qquad  \qquad&d_\sigma(y_1) =&x_1^3x_2 \qquad  & &    \\
&d_\sigma(x_2)  =0                 &    d_\sigma(y_2)=&x_1^2x_2^2   &d_\sigma(z) =& x_1^{15}+x_2^{12} \\
&d_\sigma(x_v)  =0,\,  v \in V                             &d_\sigma(y_3) =&x_1x_2^3                 & {d_\sigma (z_v)} =&{x_v^3+\sum_{(v,w)\in E}x_vx_wx_2^4, \,v \in V}.
\end{alignat*}
Therefore, the cohomology of $(\Lambda W, d_\sigma)$  is finite-dimensional because
$d_\sigma (z x_1^2 - y_2 x_2^{10}) =x_1^{17} $,  $d_\sigma ( z x_2 - y_1 x_1^{12}) =x_2^{13} $, and the cohomology class $ \left[x_v^3 \right]^4 = [-\underset{(v,w)\in E}{\Sigma}x_vx_wx_2^4]^4 = 0.$ Now, the formal dimension of  $(\Lambda W, d)$ is immediately obtained by Equation (\ref{dimensiontop}).
\end{proof}

Our next step is to describe $\Hom(\M_\G,\M_\G)$. Actually, it is the most demanding task in this paper. Recall that an automorphism of $\G$ is
a permutation $\sigma $ on $V$ with $(v, w) $ in $E$ if and only if $(\sigma (v), \sigma (w))
\in E$ for every $(v, w)$ in $E$.
 The following is  a straightforward result.
\begin{lemma}\label{lem1}  Every $\sigma \in\Aut(\G)$ induces an automorphism $f_\sigma$ of $\M_\G$.
\end{lemma}
\begin{proof}
Take $f_\sigma: \M_\G \rightarrow \M_\G $ defined by
\begin{alignat*}{2}
 f_\sigma(\omega) & =  \omega, && \qquad \, \omega \in  \{x_1, x_2, y_1, y_2, y_3, z \}\\
 f_\sigma(x_v) & = x_{\sigma(v)},&&\qquad    \,v \in V      \\
 f_\sigma(z_v) & = z_{\sigma(v)},&& \qquad  \, v \in V.
 \end{alignat*}
\end{proof}

\begin{lemma} \label{lem2} For every $f \in\Hom(\M_\G,\M_\G)$ one of the following holds.
\begin{enumerate}[label={\rm (\arabic{*})}]
\item\label{lem2-3} If $f$ is an automorphism, then there exists  $\sigma \in \Aut (\G)$ such that
 \begin{alignat*}{2}
  f(\omega) & =  f_\sigma(\omega),&&\qquad    \,\omega\in \{x_1, x_2, y_1, y_2, y_3, x_v  \mid  v \in V \}      \\
 f(z) & = f_\sigma(z)  + d(m_z) ,&& \qquad m_z  \in\M_\G^{118}\\
 f(z_v) & = f_\sigma(z_v)  + d(m_{z_v}) ,&& \qquad v \in V , \;  m_{z_v} \in \M_\G^{118}.
\end{alignat*}
\item\label{lem2-2} If $f$ is not an automorphism, then there exist $s \in\{0,1\}$ and $f_s \in\Hom(\M_\G,\M_\G)$ defined by
\begin{alignat*}{2}
f_s(\omega) & =  s\omega, && \qquad    \,\omega \in  \{x_1, x_2, y_1, y_2, y_3, z \}\\
f_s(x_v)  & =  0,&&\qquad    \,v \in V     \\
 f_s(z_v) & = 0,&& \qquad    \,v \in V
 \end{alignat*}
such that
\begin{alignat*}{2}
  f(\omega) & =  f_s(\omega),&&\qquad    \,\omega \in  \{x_1, x_2, y_1, y_2, y_3, x_v  \mid  v \in V \}  \\
  f(z) & = f_s(z)  + d(m_z) ,&&\qquad m_z  \in\M_\G^{118} \\
  f(z_v) & = f_s(z_v)  + d(m_{z_v}) ,&& \qquad v \in V, \;  m_{z_v} \in \M_\G^{118}.
   \end{alignat*}
\end{enumerate}
\end{lemma}

\begin{proof} For $f  \in\Hom(\M_\G,\M_\G)$, by degrees reasoning we write
\begin{equation}\label{description-of-f}
\begin{alignedat}{2}
f(x_1) =&  a_1x_1 \\
f(x_2) =& a_2x_2  \\
 f(y_1) =&  b_1y_1 \\
  f(y_2) = &      b_2y_2  \\
   f(y_3) = & b_3y_3  \\
f(x_v) = & \sum_{w\in V} a(v,w)x_w +a_1(v)x_1^5 + a_2(v) x_2^4, \; \;  v \in V\\
f(z) = & cz +\sum_{w\in V} c(w)z_w \\
& +\alpha_1y_1x_1^2x_2^7 +\beta_1y_2x_1^3x_2^6+\gamma_1y_3x_1^4x_2^5\\
 &+\alpha_2y_1x_1^7x_2^3 +\beta_2y_2x_1^8x_2^2+\gamma_2y_3x_1^9x_2\\
 &+ \sum_{w\in V}x_w\big(\alpha_3(w)y_1x_1^2x_2^3 +\beta_3(w)y_2x_1^3x_2^2+\gamma_3(w)y_3x_1^4x_2\big)\\
 f(z_v) =& e(v)z + \sum_{w\in V} c(v,w)z_w \\
& +\alpha_1(v)y_1x_1^2x_2^7 +\beta_1(v)y_2x_1^3x_2^6+\gamma_1(v)y_3x_1^4x_2^5\\
& +\alpha_2(v)y_1x_1^7x_2^3 +\beta_2(v)y_2x_1^8x_2^2+\gamma_2(v)y_3x_1^9x_2\\
& + \sum_{w\in V}x_w\big(\alpha_3(v,w)y_1x_1^2x_2^3 +\beta_3(v,w)y_2x_1^3x_2^2+\gamma_3(v,w)y_3x_1^4x_2\big), \; \;  v \in V.\\
\end{alignedat}
\end{equation}

Since $df(y_i)=f(dy_i)$, for $i=1,2,3$ we obtain
\begin{alignat}{2}\label{arklupfirst}
b_1=&a_1^3a_2 \qquad
b_2=&a_1^2a_2^2\qquad
b_3=&a_1a_2^3.
\end{alignat}

Since $df(z)=f(dz)$, the two expressions below must be equal
\begin{alignat*}{2}
df(z)=&c(x_1^4x_2^2y_1y_2-x_1^5x_2y_1y_3+x_1^6y_2y_3+x_1^{15}+x_2^{12})\\
&+\sum_{w\in V} c(w)\big(x_w^3 + \sum_{(w,u)\in E}x_wx_ux_2^4\big)\\
&+(\alpha_1+\beta_1+\gamma_1)x_1^5x_2^8\\
&+ (\alpha_2+\beta_2+\gamma_2)x_1^{10}x_2^4\\
&+\sum_{w\in V}(\alpha_3(w)+\beta_3(w)+\gamma_3(w))x_wx_1^5x_2^4, \\
f(dz)=& b_1b_2a_1^4a_2^2y_1y_2x_1^4x_2^2-b_1b_3a_1^5a_2y_1y_3x_1^5x_2\\
&+b_2b_3a_1^6y_2y_3x_1^6+a_1^{15}x_1^{15}+a_2^{12}x_2^{12}.
\end{alignat*}
Hence, we obtain
\begin{equation}\label{arklupsecond}
\begin{alignedat}{2}
c=& a_1^{15}=a_2^{12}\\
c=&b_1b_2a_1^4a_2^2\\
c=&b_1b_3a_1^5a_2\\
c=&b_2b_3a_1^6\\
\end{alignedat}
\end{equation}
\begin{alignat*}{2}
c(w)=& 0, \qquad\text{ for all }w\in V\\
\alpha_i+\beta_i+\gamma_i=&0, \qquad\, i=1,2\\
\alpha_3(w)+\beta_3(w)+\gamma_3(w)=&0,  \qquad \text{ for all }w\in V.
\end{alignat*}

Equations (\ref{arklupfirst}) and (\ref{arklupsecond}) are
the same as in \cite[Example 5.1]{AL2}. Therefore
 $$a_1=a_2=b_1=b_2=b_3=c=s,\text{ with } s\in\{0,1\}.$$
This yields to
\begin{alignat*}{3}
f(x_1) =&sx_1, \qquad \qquad &f(y_1) =& sy_1,  \qquad   \qquad &  f(z) &=   s z + d\big(\beta_1y_1y_2x_2^5+\gamma_1y_1y_3x_1x_2^4\big)  \\
f(x_2) =&sx_2,                  &f(y_2) = &sy_2,           & {}     & + d\big(\beta_2y_1y_2x_1^5x_2+\gamma_2y_1y_3x_1^6\big)  \\
&                            &f(y_3) = &sy_3,                 & {} &+ \sum_{w\in V}d\big(\beta_3(w)y_1y_2x_wx_2+\gamma_3(w)y_1y_3x_wx_1\big).\\
\end{alignat*}

Assume first $s=1$. Since $df(z_v)=f(dz_v),$ the following two expressions must be equal
\begin{alignat}{2}
df(z_v)=& e(v)(y_1y_2x_1^4x_2^4-y_1y_3x_1^5x_2+y_2y_3x_1^6+x_1^{15}+x_2^{12}) \label{dfz}\\
\notag
&+ \sum_{w\in V} c(v,w)\big(x_w^3 + \sum_{(w,u)\in E}x_wx_ux_2^4\big)\\
\notag
&+(\alpha_1(v)+\beta_1(v)+\gamma_1(v))x_1^5x_2^8\\
\notag
&+ (\alpha_2(v)+\beta_2(v)+\gamma_2(v))x_1^{10}x_2^4\\
\notag
&+\sum_{w\in V}(\alpha_3(v,w)+\beta_3(v,w)+\gamma_3(v,w))x_wx_1^5x_2^4 \label{fdz}\\
f(dz_v)&= \big(\sum_{w\in V} a(v,w)x_w +a_1(v)x_1^5 + a_2(v) x_2^4\big)^3\\
\notag
&+\sum_{(v,r)\in E}\big(\sum_{w\in V} a(v,w)x_w +a_1(v)x_1^5 + a_2(v) x_2^4\big)
\big(\sum_{u\in V} a(r,u)x_u +a_1(r)x_1^5 + a_2(r) x_2^4\big)x_2^4.
\end{alignat}
Close examination of these equations yields the
following remarks. Firstly, there is no summand of type
$x_vx_wx_u$, $v\ne w\ne u\ne v$ in (\ref{dfz}), and therefore
there is at most two non trivial coefficients $a(v,w)$ in
(\ref{fdz}). As for  (\ref{dfz})  there is no summand of type
$x_w^2x_r$, there is only one non trivial coefficient $a(v,w)$
in (\ref{fdz}). Hence there is, at most, a unique summand $x_w^3$
in (\ref{fdz}) and a unique non trivial coefficient $c(v,w)$ in
(\ref{dfz}). Secondly, comparing the coefficients of
$y_1y_2x_1^4x_2^4$ and $x_1^{15}$, we obtain
\begin{equation*}
\begin{array}{c}
e(v)=a_1(v)=0.\\
\end{array}
\end{equation*}
Now, there is no term of type $x_w^2x_2^4$  in (\ref{dfz}) (the graph does not contain any loop) so we deduce that
\begin{equation*}
\begin{array}{c}
a_2(v)=0.\\
\end{array}
\end{equation*}
Finally, comparing the coefficients of $x_1^5x_2^8$,
$x_1^{10}x_2^4$ and $x_wx_1^5x_2^4,$ we obtain that
\begin{alignat*}{1}
\alpha_i(v)+\beta_i(v)+\gamma_i(v) &=0, \text{ for }i=1,2\\
\alpha_3(v,w)+\beta_3(v,w)+\gamma_3(v,w)&=0.\\
\end{alignat*}

Summarizing
\begin{alignat*}{1}
f(x_v) =& a(v,\sigma(v)) x_{\sigma(v)}\\
f(z_v)= &c(v,\sigma(v)) z_{\sigma(v)}\\
& +d\big(\beta_1(v)x_2^5y_1y_2+\gamma_1(v)x_1x_2^4y_1y_3\big)\\
& +d\big(\beta_2(v)x_1^5x_2y_1y_2+\gamma_2(v)x_1^6y_1y_3\big)\\
& +\sum_{w\in V} d\big(\beta_3(v,w)x_wx_2y_1y_2+\gamma_3(v,w)x_wx_1y_1y_3\big)
\end{alignat*}
where $\sigma$ is a self-map of $V$ and
\begin{alignat*}{1}
c(v,\sigma(v)) = &a\big(v,\sigma(v)\big)^3\text{ for all }v\in V\\
c(v,\sigma(v)) = &a\big(v,\sigma(v)\big)a\big(w,\sigma(w)\big)\text{ for all }(v,w)\in E.
\end{alignat*}
Therefore $a(v,\sigma(v))^2=a(w,\sigma(w))$ if $(v,w)\in E$. Since $\G$ is not a directed graph (which implies that if $(v,w)\in E$ then $(w,v)\in E$ too)
we deduce that $a(w,\sigma(w))^2=a(v,\sigma(v))$, and hence $a(v,\sigma(v))^4=a(v,\sigma(v))$. Moreover, since $\G$ is connected, one of the following holds
\begin{enumerate}[label=\emph{\roman*)}]
\item $a(v,\sigma(v))=c(v,\sigma(v))=0$ for all $v\in V$, which proves Lemma \ref{lem2}.\ref{lem2-2} for $s=1$.
\item $a(v,\sigma(v))=c(v,\sigma(v))=1$ for all $v\in V$. Then,
\begin{alignat*}{2}
f(x_v)=& x_{\sigma(v)}\\
f(z_v)=& z_{\sigma(v)}\\
& +d\big(\beta_1(v)x_2^5y_1y_2+\gamma_1(v)x_1x_2^4y_1y_3\big)\\
& +d\big(\beta_2(v)x_1^5x_2y_1y_2+\gamma_2(v)x_1^6y_1y_3\big)\\
& +\sum_{w\in V} d\big(\beta_3(v,w)x_wx_2y_1y_2+\gamma_3(v,w)x_wx_1y_1y_3\big).
\end{alignat*}
The self-map  $\sigma\colon V\rightarrow V$ is, in fact, an element in $\Aut(\G)$. We first show that $\sigma\in\Hom(\G,\G)$, that is, $(v,w)\in E$ if and only if $(\sigma(v),\sigma(w))\in E$. Indeed,  $(v,w)\in E$ if and only if there is a summand $x_vx_wx_2^4$ in $d(z_v)$, hence, if and
only if there  is a summand $x_{\sigma{(v)}}x_{\sigma(w)}x_2^4$ in
$f(d z_v)=df(z_v)= d(z_{\sigma(v)})$, that is,  if and only if
$(\sigma(v),\sigma(w))\in E$. Now, since for every $v\in V$, $f(dz_v)=d(z_{\sigma(v)})$,  $\sigma$ is one-to-one on the neighborhood of every vertex. Therefore $\sigma\in\Aut(\G)$ \cite[Lemma 1]{Nes}, which proves Lemma \ref{lem2}.\ref{lem2-3}.
\end{enumerate}

Assume now that $s=0$. Then
\begin{equation}\label{fdzv-2}
f(dz_v)= \big(\sum_{w\in V} a(v,w)x_w +a_1(v)x_1^5 + a_2(v) x_2^4\big)^3.
\end{equation}
Since $df(z_v)=f(dz_v),$ an argument similar to the one above, comparing \eqref{dfz} and \eqref{fdzv-2}, yields to
\begin{alignat*}{2}
f(x_v)=& 0\\
f(z_v)=& 0\\
 + &d\big(\beta_1(v)x_2^5y_1y_2+\gamma_1(v)x_1x_2^4y_1y_3\big)\\
 + &d\big(\beta_2(v)x_1^5x_2y_1y_2+\gamma_2(v)x_1^6y_1y_3\big)\\
 + &\sum_{w\in V} d\big(\beta_3(v,w)x_wx_2y_1y_2+\gamma_3(v,w)x_wx_1y_1y_3\big),
\end{alignat*}
which proves Lemma \ref{lem2}.\ref{lem2-2} for $s=0$.
\end{proof}

As we mentioned in Section \ref{intro},  isomorphism classes of minimal Sullivan algebras whose cohomology is 1-connected and of finite type are in bijection with rational homotopy types for simply connected spaces with rational homology of finite type. Also, the homotopy classes of morphisms of the corresponding minimal Sullivan algebras are in bijection with the homotopy classes of maps between the corresponding rational homotopy types.
Recall that two morphisms from a Sullivan algebra to an arbitrary commutative cochain algebra, $\phi_o, \phi_1 :  (\Lambda W, d) \rightarrow  (A, d)$ are homotopic if there exists
$H: (\Lambda W, d) \rightarrow ( A, d) \otimes (\Lambda (t, dt), d )$
such that $(id \cdot \epsilon_i) H= \phi_i, i = 0,1$, where $\deg t  = 0, \ \deg dt = 1 $, and  $d$ is the differential sending $t \mapsto dt$.
The augmentations $ \epsilon_0, \ \epsilon_1: \Lambda (t, dt) \rightarrow \mathbb Q \ \text{are  defined by} \ \epsilon_0(t) = 0, \ \epsilon_1(t) = 1 $.

\begin{lemma}\label{lem3}  For any $f \in\Hom(\M_\G,\M_\G)$, one of the following holds.
\begin{enumerate}[label={\rm (\arabic{*})}]
\item There exists $f_\sigma$ automorphism, as in {\rm Lemma} {\rm\ref{lem2}.\ref{lem2-3}}, such that $f$ is homotopic to  $ f_\sigma$.
\item There exists  $f_s$ as in {\rm Lemma} {\rm\ref{lem2}.\ref{lem2-2}}, such that $ f $ is homotopic to $f_s$.
\end{enumerate}
\end{lemma}
\begin{proof}
Follows directly from Lemma \ref{lem2}.
\end{proof}

Gathering  Lemma \ref{lem1}, Lemma \ref{lem2}, and Lemma \ref{lem3},  we have
proved the following result from which we deduce Theorem \ref{thmgraph} as a corollary.
\begin{theorem}\label{sullgraph}
Let  $\G$ be a finite connected graph with more than one vertex. Then, there
exists an elliptic minimal Sullivan algebra $\M_\G$ such that the monoid of homotopy classes of self-maps is
$$[\M_\G,\M_\G] \cong \Aut (\G) \sqcup \{f_s \colon s=0,1\} .$$ Therefore $\Aut (\G)\cong
\E (\M_\G)$.
\end{theorem}

We finish this section with some comments on the properties of
the construction above. The following,  together with Theorem \ref{Frucht},  justifies
the infinitely many rational spaces $X$ from Theorem
{\ref{thmgroup}} (see also Remark \ref{many}).

\begin{proposition}\label{rem}
Let $\G_1=(V,E)$ and $\G_2=(V',E')$ be non isomorphic graphs. Then, $\M_{\G_1}$ and $\M_{\G_2}$  are non isomorphic minimal Sullivan algebras.
 \end{proposition}
\begin{proof}
Assume that $\M_{ \G_1}$ and $\M_{\G_2}$ are
isomorphic, and let $f$ denote such an isomorphism. Since
$$\vert V \vert +2=\dim \M_{ \G_1}^{40}=\dim\M_{ \G_2}^{40}=\vert V' \vert+2,$$ we
have $\vert V \vert=\vert V' \vert$ and, without loss of generality, we may assume
that $\G_1$ and $\G_2$ have the same set of vertices $V$.
Then, the isomorphism $f$ is described by the system of equations
(\ref{description-of-f}).  Reproducing the same steps as in the
proof of Lemma \ref{lem2}, we  get that $f$ is homotopic to  $f_\sigma$,
where $\sigma$ is a permutation of $V$ such that $(v,w)\in E$ if
and only if $(\sigma(v),\sigma(w))\in E'$. That is, $\sigma$
induces an isomorphism between $\G_1$ and $\G_2$.
\end{proof}

The construction of $\M$ is functorial when considering the appropriate category of graphs. Recall that given $\G_1=(V,E)$ and $\G_2=(V',E')$, a morphism $\sigma\colon\G_1\rightarrow\G_2$ is said to be full if for every pair of vertices $v,w\in V$, $(v,w)\in E$ if and only if $(\sigma(v),\sigma(w))\in E'$.

\begin{remark}\label{contrafunctor}
Let ${\G}raph_{fm}$ be the category whose objects are finite graphs with more than one vertex, and morphisms are full graph monomorphisms. Then, the construction $\M$ provides a contravariant faithful functor which is injective on objects (an embedding) from ${\G}raph_{fm}$ to the category of Sullivan algebras.
Let $\G_1=(V,E)$ and $\G_2=(V',E')$ be  graphs, and $\M_{\G_1}$
and $\M_{\G_2}$ be the associated minimal Sullivan algebras provided
by {\rm Theorem \ref{sullgraph}}. If $\sigma\colon\G_1\rightarrow\G_2$ in ${\G}raph_{fm}$, then there is a morphism of minimal Sullivan algebras
$\M(\sigma)\colon\M_{\G_2}\to\M_{\G_1}$ given by
\begin{equation*}\label{functor}
\begin{array}{rl}
\M(\sigma)(x_1) =& x_1\\ \M(\sigma)(x_2) =& x_2\\ \M(\sigma)(y_1) =& y_1\\ \M(\sigma)(y_2) =& y_2\\
\M(\sigma)(y_3) =& y_3\\ \M(\sigma)(x_{v'})=&  \begin{cases} x_v&\text{ if }\sigma(v)=v'\\
0&\text{ otherwise}\end{cases}\\ \M(\sigma)(z) = & z \\ \M(\sigma)(z_{v'})= &
\begin{cases} z_v&\text{ if }\sigma(v)=v'\\
0&\text{ otherwise.}\end{cases}
\end{array}
\end{equation*}
If $\G_1=\G_2$, then $\sigma\in\Aut(\G_1)$ and $\M(\sigma)=f_{\sigma^{-1}}$ as described in {\rm Lemma \ref{lem1}}.
\end{remark}

We finish this section illustrating other possible constructions of minimal Sullivan algebras for a given graph $\G$.

\begin{remark}\label{many}
The construction of $\M_\G$ is not unique, that is, given a finite
connected graph $\G$ with more than one vertex, there exist infinitely many non isomorphic minimal Sullivan
algebras whose group of self-homotopy equivalences is
isomorphic to $\Aut(\G)$. In fact, given a non trivial vector $(u_1,u_2)\in\Q^2$,
it is possible to construct a minimal Sullivan algebra $(\M_{(u_1,u_2)},  d_{(u_1,u_2)})$ having the same generators as $\M_\G$ and, $d_{(u_1,u_2)}$
equals $d$ in every generator but in $${
d_{(u_1,u_2)}}(z_v)=x_v^3+\sum_{(v,w)\in E}x_vx_w(u_1x_1^5 + u_2x_2^4).$$

\end{remark}

\section{From Graphs to Inflexible Manifolds}\label{inflexibledgca}

Inflexibility for an oriented compact closed  manifold $M$ means that the set of mapping degrees ranging over all continuous self-maps is finite.
By composition of self-maps it is obvious that it is equivalent to demanding that all its self-maps  have degree $-1, 0, \text{ or } 1$.
For an elliptic (hence Poincar\'e duality)  Sullivan algebra $(\Lambda W, d)$ of formal dimension $n$, inflexibility means that, for every
 $f \in \Hom ( (\Lambda W, d) ,(\Lambda W, d)  )$ and, for $x$ a representative of the fundamental class in  $H^{n} ( \Lambda W, d ) $,
 the cohomology class $[f(x)] =a [x] $, with $a \in \{-1, 0, 1\}$.

\begin{proposition}\label{tilde}
Let $\mathcal A = (\Lambda W, d)$ be a $1$-connected elliptic Sullivan algebra of formal dimension $2n$. Choose $x \in \mathcal A^{2n}$ representing the fundamental class in  $H^{2n} (\A)$. Define the Sullivan algebra
 $\widetilde \A =  (\Lambda W \otimes \Lambda(y),  \tilde d) $ with
 $\tilde d \vert_W = d$, $\deg y  = 2n-1$, and $ \tilde d (y) = x.$
Then,  $\widetilde \A$ is a $1$-connected elliptic Sullivan algebra of formal dimension $4n-1$.
Moreover, if we choose $z \in \A^{4n-1}$ such that
$d(z) = x^2$ then, $xy-z$ is a representative of the fundamental class in  $H^{4n-1} (\widetilde\A)$.
\end{proposition}
\begin{proof} First,  notice that since $(W \oplus \Q y)^{\text{even}} = W^{\text{even}}$, every element in $H^\ast (\widetilde \A)$ is nilpotent because every element in $H^\ast ( \A)$ is nilpotent.  Hence
$\widetilde \A$ is elliptic and the formal dimension is easily obtained by Equation (\ref{dimensiontop}).

Now, $\widetilde d (xy-z ) =  x \widetilde d(y) -d (z) =  0$. Let us see that it is not a boundary. Assume that $xy-z = \widetilde d ( \omega)$ for $\omega  =  \omega_1 y  + \omega_2 \in \widetilde \A ^{4n-2}$, $\omega_1, \omega_2 \in \A$.
Then, $xy-z = d (\omega_1) y  + \omega_1x + d(\omega_2)$. Since $z, \omega_1 x, d(\omega_2) \in \A$, we deduce that
$xy = d (	\omega_1) y$, and so $x = d( \omega_1)$. This contradicts the fact that $x $ is a representative of the fundamental class.
\end{proof}

\begin{lemma}\label{inflexible} The elliptic Sullivan algebra $\widetilde \A$ is inflexible if $\A $ is inflexible.  Moreover, $[\A, \A] \cong [\widetilde \A, \widetilde \A]$ as monoids and, in particular, $\E (\A) \cong \E (\widetilde \A)$.
\end{lemma}
\begin{proof}

For $\widetilde f \in \Hom (\widetilde \A, \widetilde \A)$, let $f$ denote $\widetilde f \vert_\A$. Then $\widetilde f(xy-z) = f(x)\widetilde f (y) - f(z)$.
Since $\A$ is inflexible, $f(x) = ax  + d({m_x})$ with $a \in \{1, 0, -1\}$. Applying $d$ to $f(z)$, and using $d(z) = x^2$, again a straightforward calculation shows that $$f(z) = a^2z + (2axm_x + m_xd(m_x)) + d(\gamma).$$ Applying now $\widetilde d$ to $\widetilde f (y)$, and using
$\widetilde d (y) = x$, a straightforward calculation shows that $$\widetilde f (y) = ay + m_x + d(\gamma').$$  Hence
$[\widetilde f (xy-z) ] = a^2 [xy-z]$, proving that  $\widetilde \A$ is inflexible.

 Observe now that any $\widetilde f \in \Hom (\widetilde \A, \widetilde \A)$ is determined,  up to homotopy,  by its rectriction to $\A$.
The only undetermined term appears when  $\widetilde f (y)$ is computed. This means that if $\widetilde f_1 \vert_\A $ and $ \widetilde f_2 \vert_\A $ are equal, then $\widetilde f_1 (y) - \widetilde f_2 (y) = d(\gamma'_1 - \gamma'_2). $ Hence $\widetilde f_1$ and $\widetilde f_2$ are homotopic.  The same way, any $f  \in  \Hom ( \A, \A)$ can be extended to $ \Hom (\widetilde \A, \widetilde \A)$ in, up to homotopy, a unique way.
\end{proof}

We can now prove Theorem \ref{realizationinflex}.
\begin{proof}[Proof of Theorem \ref{realizationinflex}]
Let $G$ be a finite group. There exists a finite and connected graph $\G = (V, E)$ such that $\Aut (\G) \cong G$ (Theorem \ref{Frucht}).
 Associated to the graph $\G$ of order $n$,  there exists a $1$-connected elliptic minimal Sullivan algebra $\M_\G$ (of formal dimension
$208 + 80n$) such that
$\Aut (\G)  \cong \E (\M_\G)$ (Theorem \ref{sullgraph}). We modify $\M_\G$ into
an elliptic minimal Sullivan algebra $\widetilde \M_\G$ of formal dimension $(416 + 160n )- 1$ (Proposition \ref{tilde}) which,
by Lemma \ref{inflexible} is inflexible since $\M_\G$ is inflexible. This is clear  since $[\M_\G, \M_\G]  \cong G \sqcup\{f_0, f_1\}$ is finite, and because
of the multiplicativity of the mapping degree.

Now, since $415 + 160n  \not \equiv 0 \pmod{4}$,  the theorems of Sullivan  \cite[Theorem (13.2)]{Su} and Barge \cite[Th\'eor\`eme 1]{Ba} give a sufficient condition for the realization of $\widetilde \M_\G$ by a simply connected manifold $M$.

Finally,  again by Lemma \ref{inflexible}, $\E (\M_\G) \cong \E(\widetilde \M_\G).$
Hence, putting the isomorphisms of groups altogether, we get
$$G \cong \Aut (\G) \cong \E (\M_\G) \cong \E(\widetilde \M_\G) \cong \E (M_0),$$
where $M_0$ is the rational homotopy type of $M$.
\end{proof}

The question of whether certain orientation-reversing maps on manifolds exist is treated in the literature (see for example \cite{Pup, Am}). Examples of these manifolds
are provided by Theorem~\ref{realizationinflex}.
\begin{corollary} For any $n >1$, there exists a simply connected manifold $M$ of dimension $415 + 160n$ that does not admit a reversing orientation self-map.
\end{corollary}
\begin{proof} The existence of such a manifold $M$ is given by Theorem \ref{realizationinflex} for a graph $\G$ of order~$n$, with $\widetilde \M_\G$ being the minimal Sullivan algebra of $M$.  Now, any self-map of $\widetilde \M_\G$, is shown to verify $\deg (\widetilde f) = \deg (\widetilde f \vert_{\M_\G})^2$ (proof of Lemma \ref{inflexible}). Therefore, any self-map of $M$ has either degree  $0$ or $1$.
\end{proof}



\begin{thebibliography}{99}

\bibitem{Am}  M.\ Amann, \emph{Mapping degrees of self-maps of simply-connected manifolds}, Preprint arXiv:1109.0960v1.

\bibitem{A2} M.\ Arkowitz, \emph{The group of self-homotopy equivalences--a survey, in: Groups of self-equivalences and related topics},
 Lecture Notes in Math., Springer, \textbf{1425}  (1990), 170--203.

\bibitem{A1} M.\ Arkowitz, \emph{Problems on self-homotopy equivalences, in: Groups of homotopy self-equivalences and related topics},
 Contemp. Math.\  \textbf{274} (2001), 309--315.

\bibitem{AL2} M.\ Arkowitz, G.\ Lupton, \emph{Rational obstruction theory and rational homotopy sets},
 Math.\ Z.\ \textbf{235}  (2000),  no.\ 3, 525--539.

\bibitem{Ba} J.\ Barge, \emph{Structures diff\'erentiables sur les types d'homotopie rationnelle simplement connexes}, Ann.\ Sci.\ \'Ecole Norm.\ Sup.\ ({4}),  \textbf{9} (1976), no.\ 4, 469--501.

\bibitem{BM} M.\ Benkhalifa, \emph{Rational self-homotopy equivalences and Whitehead exact sequence},
 J.\ Homotopy  Relat.\ Struct.\ \textbf{4} (2009), no.\ 1, 111--121.

\bibitem{BM2} M.\ Benkhalifa, \emph{Realizability of the group of rational self-homotopy  equivalences},
J.\ Homotopy  Relat.\ Struct.\ \textbf{5} (2010), no.\ 1, 361--372

\bibitem{bollo} B.\ Bollob\'as, Modern graph theory. Graduate Texts in Mathematics, \textbf{184}. Springer-Verlag, New York, 1998.

\bibitem{C-S} A.H.\ Copeland, A.O.\ Shar, \emph{Images and pre-images of localization maps}, Pacific J.\ Math.\ \textbf{57} (1975), no. 2, 349--358.

\bibitem{CV3} C.\ Costoya, A.\ Viruel, \emph{Faithful actions on Differential Graded Algebras determine the isomorphism type of a large class of groups}, Preprint arXiv:1206.3639.

\bibitem{CL} D.\ Crowley, C.\ L{\"o}h, \emph{Functorial semi-norms on singular homology and (in)flexible manifolds}, Preprint arXiv:1103.4139v1.

\bibitem{FF} F.\ Federinov, Y.\ F\'elix, \emph{Realization of $2$-solvable nilpotent groups as groups of classes of homotopy self-equivalences},
Topology Appl.\ \textbf{154} (2007), no.\ 12, 2425--2433.

\bibitem{F1} Y.\ F\'elix, La dichotomie elliptique-hyperbolique en homotopie
rationnelle. Ast\'erisque,  \textbf{176} (1989).

\bibitem{F2} Y.\ F\'elix,  \emph{Problems on mapping spaces and related subjects, in: Homotopy theory of function spaces and related topics},
Contemp.\ Math.\  \textbf{519}  (2010), 217--230.

\bibitem{FHT2} Y.\ F\'elix, S.\ Halperin, J.C.\ Thomas,  Rational homotopy theory. Graduate Texts in Mathematics, \textbf{205}, Springer-Verlag, New York, 2001.

\bibitem{Frucht1} R.\ Frucht, \emph{Herstellung  von Graphen mit vorgegebener abstrakter Gruppe},
Compositio Math.\ \textbf{6} (1939), 239--250.

\bibitem{Frucht2} R.\ Frucht, \emph{Graphs of degree three with a given abstract group},
Canad.\ J.\ Math.\ \textbf{1} (1949), 365--378.

\bibitem{Gromov} M.~Gromov, \emph{Metric structures for Riemannian and non-Riemannian spaces} with appendices by
M.~Katz, P.~Pansu and S.~Semmes, translated from the French by Sean Michael Bates.
Progress in Mathematics \textbf{152}, Birkh\"auser, 1999.

\bibitem{Hal}   S.\ Halperin,  \emph{Finiteness in the minimal models of Sullivan},
Trans.\ Amer.\ Math.\ Soc.\  \textbf{230} (1977),
173--199.

\bibitem{ka1} D.\ Kahn, \emph{Realization problems for the group of homotopy classes of self-equivalences},
Math.\ Annal.\ \textbf{220} (1976), no.\ 1,  37--46.

\bibitem{ka3} D.\ Kahn, \emph{Some research problems on homotopy-self-equivalences, in: Groups of self-equivalences and related topics},
Lecture Notes in Math.\ \textbf{1425}, Springer, Berlin, (1990), 204--207.

\bibitem{Ma} K.\ Maruyama,  \emph{Finite complexes whose self-homotopy equivalence groups realize the infinite cyclic group},
 Canad.\  Math.\ Bull.\  \textbf{37} (1994), no. 4, 534--536.

\bibitem{Nes} J.\ Ne{\v{s}}et{\v{r}}il, \emph{Homomorphisms of derivative graphs}, Discrete Math.\ \textbf{1} (1971/72), no.\ 3, 257--268.

\bibitem{Oka} S.\ Oka, \emph{Finite complexes whose self-homotopy equivalences form cyclic groups},
  Mem.\ Fac. \ Sci.\  Kyushu Univ.\ Ser.\ A.\ \textbf {34} (1980), no. 1, 171--181.

\bibitem{Pup} V.\ Puppe, \emph{Simply connected $6$-dimensional manifolds with little symmetry and algebras with small tangent space}. In Prospects
  in topology (Princeton, NJ, 1994), Ann.\ of Math.\ Stud.\ \textbf{138}, 283--302.

\bibitem{Ru} J.\ Rutter, \emph{Spaces of homotopy self-equivalences. A Survey},
Lecture Notes in Mathematics, \textbf{1662}, Springer-Verlag, Berlin, 1997.

\bibitem{Su} D.\ Sullivan, \emph{Infinitesimal computations in topology}, Inst.\ Hautes \'Etudes Sci.\ Publ.\ Math.\ \textbf{47} (1977), 269--331.

\end{thebibliography}
\end{document}